\documentclass[11pt,reqno]{amsart}
\usepackage{amssymb, amsfonts, amsbsy, latexsym}
\usepackage[english]{babel}

\newtheorem{thm}{Theorem}[section]
\newtheorem{theorem}[thm]{Theorem}

\newtheorem{corollary}[thm]{Corollary}

\newtheorem{lemma}[thm]{Lemma}
\newtheorem{proposition}[thm]{Proposition}

\newtheorem{definition}[thm]{Definition}
\theoremstyle{remark}
\newtheorem{conjecture}[thm]{Conjecture}
\newtheorem{remark}[thm]{Remark}

\newtheorem{problem}[thm]{Problem}

\newcommand{\RR}{\mathbb R}

\newcommand{\cH}{\mathcal H}

\newcommand{\PP}{\mathcal P}
\newcommand{\QQ}{\mathcal Q}
\def\cW{{\mathcal W}}
\def\cT{{\mathcal T}}

\def\cN{{\mathcal N}}
\def\cP{{\mathcal P}}
\def\cR{{\mathcal R}}
\def\cSw{{\mathcal S}_{\cW}}
\def\cPw{{\mathcal P}_{\cW}}
\def\Sw{S_{\cW}}
\def\II{\mathbb I}

\def\<{\langle}
\def\>{\rangle}
\def\ds{\displaystyle}
\def\sp{\overline{sp}}
\def\tab{\ \ \ }
\def\be{\begin{equation}}
\def\ee{\end{equation}}

\title[Non-orthogonal Fusion Frames]
{Non-orthogonal fusion frames and the sparsity of fusion frame operators}

\author[Cahill, Casazza, Li]{Jameson Cahill,  Peter G. Casazza, Shidong Li}
\date{\empty}
\address{Cahill/Casazza, Department of Mathematics,
University of Missouri, Columbia, MO 65211-4100}

\address{Shidong Li, Department of Mathematics, San Francisco State University, 
San Francisco, CA 94132}
\email{jameson.cahill@gmail.com; casazzap@missouri.edu; shidong@sfsu.edu}

\begin{document}

\begin{abstract}
Fusion frames have become a major tool in the implementation of distributed systems.
The effectiveness of fusion frame applications in distributed systems is reflected in the efficiency of the end fusion process.  This in turn is reflected in the efficiency of the inversion of the fusion frame operator $S_{\cW}$, which in turn is heavily dependent on the sparsity of $S_{\cW}$.
We will show that sparsity of the fusion frame operator
naturally exists by introducing a notion of {\it non-orthogonal fusion frames}.
We show that for a fusion frame $\{W_i,v_i\}_{i\in I}$,
if $\text{dim}(W_i)=k_i$, then the matrix of the non-orthogonal fusion frame operator
$\cSw$ has in its corresponding location at most a $k_i\times k_i$ block matrix.   We provide necessary and sufficient conditions for which the new fusion frame operator $\cSw$ is diagonal and/or a multiple of an identity.  A set of other critical questions are also addressed.  A scheme of {\it multiple fusion frames} whose corresponding fusion frame operator becomes an diagonal operator is also examined.
\end{abstract}

\subjclass[2000]{94A12, 42C15, 68M10, 68Q85}
\date{April 18, 2010}
\keywords{Frames, Fusion Frames, Sparsity of the Fusion Frame Operator, Sensor Network, Data Fusion, Distributed Processing, Parallel Processing.}
\thanks{The first author is supported by NSF DMS 1008183; the
second author is supported by NSF DMS 1008183
DTRA/NSF 1042701, AFOSR F1ATA00183G003; The third author is supported by NSF DMS 0709384 and NSF DMS 1010058}

\maketitle
%
%

\section{Introduction}

Fusion frames were introduced in \cite{CK04} (under the name {\it frames of subspaces}) and \cite{CKL08}, and have quickly turned into an industry (see www.fusionframes.org).  Recent developments include applications to sensor networks \cite{CKLR07}, filter bank fusion frames
\cite{CFM10}, applications to coding theory \cite{BK09}, compressed sensing \cite{BKR09}, construction methods \cite{CCHKP10, CCHKP09,CFMWZ10,CFMWZ09,CFMWZ10A}, sparsity for fusion frames \cite{CHKK10}, and frame potentials and fusion frames \cite{MRS10}.
Until now, most of the work on fusion frames has centered on
developing their basic properties and on {\it constructing}
fusion frames with specific properties.  We now know that there are very few tight fusion
frames without weights.  For example, in \cite{CFMWZ10} the authors classify all
triples $(K,L,N)$ so that there exists a tight fusion frame $\{W_i\}_{i=1}^K$
with $dim\ W_i = L$, for all $i=1,2,\ldots,K$ in ${\mathcal H}_N$.

A major stumbling block
for the application of fusion frame theory is that in practice, we generally do not get to
{\it construct} the fusion frame, but instead it is thrust upon us by the application.
In a majority of fusion frame applications, such as in sensor network data processing, each sensor spans a fixed subspace $W_i$ of ${\mathcal H}$
generated by the spatial reversal and the translates of the sensor's impulse response function \cite{LiYan08}, \cite{LiYaoYi2008}. There is no opportunity then
for subspace transformation, manipulation and/or selection.  As a result, the fusion frame operator $S_{\cW}$ is always non-sparse with an extremely high probability.  The lack of sparsity of $\Sw$ is
a significant hinderance in computing $\S_W$
and its inverse, which is necessary to apply the theory.
So the central issue in the effective application of fusion frames is to have sparsity
for the fusion frame operator - preferably for it to be a diagonal operator.

We have long suspected that there has to be a way to ensure that $\Sw$ is no more than a block diagonal operator with each block having the dimension of the corresponding subspace.   It turns out that a notion of non-orthogonal fusion frames achieves that and this is the central theme of
this paper.

\section{Non-orthogonal fusion frames}

Nonorthogonal fusion frames are a modification of fusion frames \cite{CK04}, \cite{CKL08} with a sequence of non-orthogonal projections operators.  A non-orthogonal projection onto $W$ is a linear mapping
$\mathcal{P}_{W}$ from $\mathcal{H}$ onto $W$ which satisfies $\mathcal{P}_{W}^2=\mathcal{P}_{W}$.
An important property is that the adjoint $\mathcal{P}^{*}_{W}$ is also a non-orthogonal
projection from $\mathcal{H}$ onto $\mathcal{N}(\mathcal{P}_{W})^{\perp}$ with $W^{\perp}$
being the null space (of $\cP^*$).  Here
$\mathcal{N}(\mathcal{P}_{W})=\{f\in\mathcal{H}:\mathcal{P}_{W}f=0\}$.  Also observe that
we must have that $\mathcal{N}(\mathcal{P}_{W})\cap W=\{0\}$, \text{i.e.},
$\mathcal{N}(\mathcal{P}_{W})\oplus W=\mathcal{H}$.

\begin{definition}
Let $I$ be a countable index set.  Let $\{W_i\}_{i\in I}$ be a family of closed subspaces in $\cH$, and let $\{v_i\}_{i\in I}$ be a family of positive weighting scalers.  Denote by $\cP_i$ a non-orthogonal projection onto $W_i$.  Then $\{\cP_i, v_i\}_{i\in I}$ is a {\it non-orthogonal fusion frame} of $\cH=\overline{span}(\sum_{i\in I} W_i)$ if there are constants $0 < C\leq D < \infty$ such that
\begin{equation}\label{nonog-FF}
\forall f\in\cH, \ \ \ C\|f\|^2\leq \sum_{i\in I} v_i^2\|\cP_i f\|^2 \leq D\|f\|^2.
\end{equation}
\end{definition}

\noindent {\bf Remarks}: \
(1)  Throughout this paper we will use $\pi$ for an orthogonal projection.
It is obvious that if $\cP_i$ is an orthogonal projection $\pi_i$, then our notion
of a non-orthogonal fusion frame becomes the standard fusion frame.

(2)  In general, let $\ds(\sum \oplus W_i)_{l^2} \equiv \{\{f_i\} | f_i\in W_i \ \text{and}\  \|f_i\|\in l^2(I)\}$.
 Define the analysis operator $\ds{\cT}_{\cW} :\cH\rightarrow (\sum \oplus W_i)_{l^2}$ by
\[
  \cT_{\cW} f =\{v_i\cP_i f\}_{i\in I},\mbox{ for all } f\in \cH.
\]
Then
\[
 \cT_{\cW}^* f =\sum_{i\in I} v_i\left(\cP_i^*\cP_i\right) f,\mbox{ for all }f\in \cH.
\]

The new (non-orthogonal) fusion frame operator $\cSw :\cH \rightarrow \cH$ becomes
\[
 \cSw \equiv \cT_{\cW}^*\cT_{\cW} = \sum_{i\in\II} v_i^2\cP_i^*\cP_i.
\]
We compare this to the standard fusion frame operator
\[
 \cSw \equiv \cT_{\cW}^*\cT_{\cW} = \sum_{i\in\II} v_i^2\pi_i.
\]
It is also true that the non-orthogonal fusion frame condition (\ref{nonog-FF}) is equivalent to that
\[
 C\,Id \leq \cSw  \leq D\, Id
\]

(3)  \ If the standard (orthogonal) fusion frame (OGFF) condition \cite{CK04}, \cite{CKL08} holds, there will be no loss of information with non-orthogonal projection operators.   Instead, there are infinitely many flexibilities
now available which is highly beneficial to the sparsity of $\cSw$ as we demonstrate next.

(4)  Oftentimes, subspaces $\{W_i\}$ are given a priori by applications.  Subspace manipulation does not exist nor is allowed in those applications.  As a result, the fusion frame operator $\Sw$ given by the orthogonal projections are fixed, and are non-sparse with probability nearly 1.  For instance, in $\RR^3$, let $\ds W_1=\{z=0\}$, and $\ds W_2=\{x+y+z=0\}$ be two planes.   $\Sw$ by the OGFF definition \cite{CK04}, \cite{CKL08} $\ds\Sw = \pi_1+\pi_2 $ gives rise to a full matrix with no zero entry.

If on the other hand, if we take $\ds \cP_1=\pi_1$, but let $\ds \cP_2$ be the non-orthogonal
projection with $\ds\cN(\cP_2)=\{z=0\} \cap \{y=0\}$ so that
\[
\cP_2 =\left(\begin{array}{ccc}
  0 & -1 & -1 \\
  0 & 1 & 0 \\
  0 & 0 & 1\end{array}\right),
\]

then $\cP_2^*\cP_2$ is a $2\times 2$ block matrix
\[
\cP_2^*\cP_2 =\left(\begin{array}{ccc}
  0 & 0 & 0 \\
  0 & 3 & 1 \\
  0 & 1 & 2\end{array}\right).
\]

Also, the corresponding non-orthogonal fusion frame operator $\cSw$ has the standard matrix representation
\[
\cSw =\pi_1 + \cP_2^*\cP_2 =\left(\begin{array}{ccc}
 1 & 0 & 0 \\
 0 & 3 & 1 \\
 0 & 1 & 2\end{array}\right),
\]
which is now a relatively sparse representation - already
much better than that of the orthogonal projections.
\smallskip
\smallskip

\noindent{\bf Diagonal $\cSw$}:  \ \ One can achieve more in this example with non-orthogonal projections.  Say, if we take $\ds \cN(\cP_1)=span\{e_2 +e_3\}$.  Then $\ds \cN(\cP_1)\cap W_1={0}$, and
\[
\cP_1 =\left(\begin{array}{ccc}
  1 & 0 & 0 \\
  0 & 1 & -1 \\
  0 & 0 & 0\end{array}\right).
\]
Then
\[
\cP_1^*\cP_1 =\left(\begin{array}{ccc}
  1 & 0 & 0 \\
  0 & 1 & -1 \\
  0 & -1 & 1\end{array}\right).
\]
Consequently,
\[
\cSw =\cP_1^*\cP_1 + \cP_2^*\cP_2 =\left(\begin{array}{ccc}
 1 & 0 & 0 \\
 0 & 3 & 0 \\
 0 & 0 & 3\end{array}\right),
\]
which yields a diagonal non-orthogonal fusion frame operator $\cSw$.   This situation is highly beneficial to all fusion frame applications.

(5) \ Suppose that fusion frames are used in sensor network applications.  Each subspace $W_i$ represents a sensor.   The measurement of each sensor is a typical frame expansion $\{\< f, w_n\>\}$ \cite{LiYan08}.  Therefore, not only the subspaces $\{W_i\}$ are fixed by the sensors in the network, but also the sensor measurements are given a priori.  So diagonalizing $S_{\cW}$ through subspace transformations and/or rotations are not permitted.

Fortunately, non-orthogonal projections make use of the given sensor measurement (or the sensory frame expansion) precisely and naturally using the notion of pseudoframes for subspaces (PFFS) \cite{LiOgPFFS98}.  In Section \ref{secPFFS} we consider the implementation of non-orthogonal projections $\cP_i$ using PFFS.

(6) Relating to the notion of nonorthogonal fusion frames,
there is the notion of g-frames \cite{sun2006g}, \cite{sun2007stability}.  Actually,
g-frames are more general classes of ``operator frames''.  Though nonorthogonal fusion frames are a class of g-frames with projection operators, the study of this (nonorthogonal) projection class has never been carried out, and the restriction to (nonorthogonal) projection operators also makes the analysis less flexible than that of the general operator frames.  Yet, it is this class of projection operators that actually find realistic applications in sensor array or distributed system data fusion.  Because projection operators really have the physical interpretation, namely, signals measured by sensors are really projections of the original signal/function onto the subspace W spanned by the sensor. Linear measurements of a signal by sensors and/or linear devices are typically modeled by an orthogonal projection operator.  Sensors and/or linear devices can also function in a nonorthogonal way, the principle of which is discussed in detail in {\em Section} \ref{secPFFS}.

Our work here has also led to a synthesis of positive and self-adjoint operator $T$ by projections $\cP_{ij}$ in the form of $\sum_{ij}v_i^2\cP^*_{ij}\cP_{ij}$.  These ideas will be developed in a later article.

\section{Main Problem Statements}

We list here some of the problems needed to be resolved in the topic of non-orthogonal fusion frames.  In this article, we provide solutions to several of these problems.

\begin{problem}[Main Problem]  \label{prob1} Given subspaces $\{W_i\}_{i\in I}$ of
$\cH_N$ which span $\cH_N$, does there exist a family of non-orthogonal
projections $\{\PP_i\}_{i\in I}$ with $\PP_i$ mapping onto $W_i$ so that
\[ \sum_{i\in I}\PP_i^*\PP_i = \lambda I?\]
Alternatively,
\[ \sum_{i\in I}\PP_i^*\PP_i =D,\ \ \mbox{$D$ a diagonal operator}?\]
\end{problem}

\begin{conjecture}\label{conj1}
We believe that Problem \ref{prob1} has a negative answer with strict diagonal right hand sides.  But, sparsity to certain degree is always achievable.
\end{conjecture}

Since non-orthogonal projections onto a given subspace are no longer unique, the following problem is very natural, and likely to have a positive answer.

\begin{problem}\label{prob2}
Given subspaces $\{W_i\}_{i\in I}$ of
$\cH_N$ which span $\cH_N$, do there exist multiple non-orthogonal
projections $\{\PP_i\}_{i\in I}$ and $\{\QQ_i\}_{i\in I}$ and weights $\{v_i\}_{i\in I}$ and $\{w_i\}_{i\in I}$  with $\PP_i,\QQ_i$ mapping onto $W_i$ so that
\[ \sum_{i\in I}(v_i^2\PP_i^*\PP_i+w_i^2\QQ_i^*\QQ_i) = \lambda I?\]
Or perhaps some number of projections - which should not be too large.
\end{problem}

\begin{remark}\label{rem1}
Problem \ref{prob2} has a positive answer with $v_i=1$ for every $i\in I$ if the subspaces all have dimension
$\ge \frac{N}{2}$.  We show this in Proposition \ref{prop3}.
\end{remark}

Since for every subspace $W$, either $W$ or $W^{\perp}$ (or alternatively, $(I-\PP)W$
for any projection $\PP$ onto $W$) has dimension $\ge \frac{N}{2}$, it would be interesting to solve the next problem.

\begin{problem}\label{prob3}
Given subspaces $\{W_i\}_{i\in I}$, weights $\{v_i\}_{i\in I}$, and projections $\{\PP_i\}_{i\in I}$ onto the
$W_i$ satisfying
\[ \sum_{i\in I}v_i^2\PP_i^*\PP_i = \lambda I,\]
does there exist projections $\{\QQ_i\}_{i\in I}$ onto $\{W_i^{\perp}\}_{i\in I}$
(or onto $\{(I-\PP_i)\cH_N\}_{i\in I}$) and weights $\{w_i\}_{i\in I}$ so that
\[ \sum_{i\in I}w_i^2\QQ_i^*\QQ_i = \mu I?\]
\end{problem}


\section{Block diagonal characterization of the non-orthogonal projection $\cP$}\label{section1}

We first show that every subspace $W$ with dim $W=k$
has a projection $\PP_W$ onto it for which the
matrix of $\PP_W^*\PP_W$ is a $k\times k$ block matrix.

\begin{proposition} \label{prop_kxk}
Let $W$ be a $k$-dimensional subspace of $\cH_N$.  Then there is a
subset $K\subset \{1,2,\ldots,N\}$ with $|K|=k$ and a projection $\PP_W$
onto $W$ so that the matrix of $\PP_W$ has non-zero entries only
on the entries of $K\times K$.
\end{proposition}

\begin{proof}
Let $\{e_i\}$ be the standard orthonormal basis of $\cH_N$.
Given an orthonormal basis $\{x_i\}_{i=1}^k$ for $W$, if we row reduce it, we
will find a set $K\subset \{1,2,\ldots,N\}$ with $|K|=k$ so that the restriction of the operator
$\pi_K:W\rightarrow V_K = span\ \{e_i\}_{i\in K}$ is invertible on $V_K$.  Define a mapping
\[
\PP_W = \left ( \pi_K|_{V}\right )^{-1}\pi_K.
\]
Then $\PP_W$ is a projection onto $W$.  Also, $\PP_We_j = 0$ if $j\notin K$
implies that for $j\notin K$ we have for all $i$:
\[
\langle \PP_W^*\PP_W e_i,e_j\rangle = \langle \PP_We_i,\PP_We_j\rangle
= \langle \PP_We_i,0\rangle =0.
\]
So the only non-zero entries in the matrix of $\PP_W^*\PP_W$ are the entries
from $K\times K$.
\end{proof}

An alternative argument of the proof goes as follows.   Since $dim(W)=k$, one can always find a set $K^{'}\subset \{1, 2, \ldots,N\}$ with $|K^{'}|=N-k$ such that $W^{'}\equiv span\{e_j\}_{j\in K^{'}}$ complements $W$, i.e., $W^{'}\cap W =\{0\}$ and $W+W^{'}=\cH_N$.   Now, set the null space of $\cP_W$ to be $\cN(\cP_W)=W^{'}$.  Then $\cP_W e_j =0$ for all $j\in K^{'}$.   The rest follows by the last 3 lines of the previous proof.   Note, there are consequently $N-k$ columns of zeros in the matrix of $\cP_W$ with respect to the orthonormal basis $\{e_j\}$.

\smallskip
\smallskip
\smallskip

In fact, more can be said about the sparsity of $\cP_W$.

\begin{proposition}
Let $\mathcal{H}$ be Hilbert space with orthonormal basis $\{e_i\}_{i=1}^N$.  Then for every subspace $W\subseteq\mathcal{H}$, there exists a projection $\mathcal{P}_W$ such that the matrix of $\mathcal{P}_W$ is triangular with respect to $\{e_i\}_{i=1}^N$.
\end{proposition}
\begin{proof}
Choose $K\subseteq\{1,...,N\}$, $|K|=k=\dim(W)$,
$K = \{i_1,i_2,\ldots,i_k\}$  and $V = span\ \{e_{i_j}\}_{j=1}^k$,
so that the orthogonal projection $\pi_V$ onto $V$ is a
bijection between $V$ and $W$.  We know there exists $x_1\in W$ so that $\pi_Vx_1=e_{i_1}$.  Let $W_1=\{w\in W:\langle w,x_1\rangle=0\}$ and $A_1=\{x\in W:\pi_Vw\in\text{span}\{e_{i_1},e_{i_2}\}\}$.  Then $A_1\cap W_1$ is a one dimensional subspace of $W$, so choose $x_2$ in this subspace.  Repeat inductively so that $\pi_Vx_j\in\text{span}\{e_{i_1},...,e_{i_j}\}$.  Now define $U:V\rightarrow W$ so that $U\pi_Vx_j=x_j$, and define $\mathcal{P}_W=U\pi_V$.  Then $\mathcal{P_W}x_j=U\pi_Vx_j=x_j$, so $\mathcal{P}_W^2=\mathcal{P}_W$.  Also, for $i\in K$ we have $e_i=\sum_{j=1}^ib_{ij}\pi_Vx_j$.  Therefore, for $i,\ell \in K$ with $\ell>i$ we have
$$
\langle\mathcal{P}_We_i,e_{\ell}\rangle=\sum_{j=1}^ib_{ij}\langle\pi_Vx_j,e_{\ell}\rangle=0.
$$
Also, if $i\notin K$ then $\mathcal{P}_W(e_i) = 0$ so for all $\ell$,
$$
\langle\mathcal{P}_We_i,e_{\ell}\rangle=0.
$$
\end{proof}

\noindent\textbf{Remark: }Since ${\mathcal P}_W^*{\mathcal P}_W$ is self adjoint, it is triangular if and only if it is diagonal.   Consequently, the triangular nature of $\cP_W$ may only result in $K\times K$ block diagonal nature in $\cP_W^*\cP_W$.  In Section \ref{sec_diag}, we will provide a characterization of when $\cP_W^*\cP_W$ can always be diagonal.

But first, let us examine an immediate consequence of the non-orthogonal fusion frame applied to conventional frames.   The evaluation of dual frames (to any conventional frames) becomes effortless.  There is a corresponding Parseval fusion fusion associated with any given conventional frame.
\smallskip
\smallskip

\noindent{\bf Example}\ \  (The case of conventional frames) \ \
Let $\{x_i\}_{i=1}^{M}$ be a conventional finite frame of $\cH_N$.  The following is immediate.

\begin{proposition}\label{prop_ConventionalFrame}
Let $\{x_i\}_{i=1}^M$ is a frame for $\mathcal{H}_N$ and let $W_i=\text{sp}\{x_i\}$.  Then there exists projections $\mathcal{P}_i$ onto $W_i$ and weights $v_i$ so that $\sum_{i=1}^Mv_i^2\mathcal{P}_i^*\mathcal{P}_i=I$.
\end{proposition}
\begin{proof}
By Theorem 3.1 there exist projections $\mathcal{P}_i$ so that the matrix of $\mathcal{P}_i^*\mathcal{P}_i$ will have only one nonzero entry $r_i$, and this entry will be on the diagonal.  Let $j_i$ denote the position of $r_i$ in the matrix of $\mathcal{P}_i^*\mathcal{P}_i$.  Now for each $k=1,...,N$ let $I_k=\{i:j_i=k\}$ and let $v_i^2=(\sum_{i\in I_k}r_i)^{-1}$ for each $i\in I_k$.
\end{proof}

%

More specifically, let us also make this statement constructively.  Write $\ds x_i=\left(x_{ij}\right)^N_{j=1}$.  Assume that the index enumeration of the first $N$ column
vectors $\{x_i\}$ is such that $|x_{jj}| > 0$ for all $1\leq j\leq N$.
For better stability, we may also assume that the enumeration of the first $N$ vectors is such that, for a given $j$, $0< |x_{jj}|$ is the largest possible among all possible index permutations (such that $|x_{jj}| > 0$ for all $1\leq j\leq N$).

Let $\{e_i\}$ be the standard ONB of $\cH_N$.  For $1\leq i\leq N$, select $\cP_i$ such that
\be\label{Pi_1}
\cN(\cP_i)=\cN(e_i).
\ee
Write $\ds\bar x_i \equiv \frac{x_i}{x_{ii}}$.  Then
\be\label{Pi_dim1}
\cP_i=\left(\begin{array}{ccccccc}
0 & \cdots & 0 & \frac{x_{i1}}{x_{ii}} & 0 & \cdots & 0 \\
\vdots & \vdots & \vdots & \vdots & \vdots & \vdots & \vdots \\
0 & \cdots & 0 & \frac{x_{i i-1}}{x_{ii}} & 0 & \cdots & 0 \\
0 & \cdots & 0 & 1 & 0 & \cdots & 0 \\
0 & \cdots & 0 & \frac{x_{i i+1}}{x_{ii}} & 0 & \cdots & 0 \\
\vdots & \vdots & \vdots & \vdots & \vdots & \vdots & \vdots \\
0 & \cdots & 0 & \frac{x_{iN}}{x_{ii}} & 0 & \cdots & 0
             \end{array} \right), \ \ \ 1\leq i \leq N.
\ee
That is, the only nonzero vector $\bar x_i$ in $\cP_i$ is at the $i^{th}$ column.  This is because $(\cP_i e_i)_i=1$ and $\cP_i e_i =\alpha x_i$.  One can verify easily that such a $\cP_i$ is a projection.  As a result, for $1\leq i\leq N$,
\[
\left(\cP^*_i\cP_i\right)_{mn}=\left\{\begin{array}{ll}
  \|\bar x_i\|^2, & m=n=i \\
  0,              & \text{otherwise}.\end{array}\right.
\]
\smallskip

For $N+1\leq i\leq M$, let $j_i\in \{1,\cdots,N\}$ be such that $|x_{i j_i}|\geq |x_{i j}| >0 $ for all $1\leq j\leq N$ (assuming none of the vectors of $\{x_i\}$ is a zero vector).  Then, define
\be \label{Pi_2}
\cN(\cP_i)=\cN(e_{j_i}),
\ee
and write $\ds \hat x_i\equiv \frac{x_i}{x_{i j_i}}$.  Then $\cP_i$ has the same expression as in
(\ref{Pi_dim1}) with all zero columns but $\hat x_i$ at the $j_i^{th}$ column.
Therefore, for $N+1\leq i\leq M$,
\[
\left(\cP^*_i\cP_i\right)_{mn}=\left\{\begin{array}{ll}
  \|\hat x_i\|^2, & m=n=j_i \\
  0,              & \text{otherwise}.\end{array}\right.
\]

Evidently, the choice of $\cN(\cP_i)=\cN(e_{j_i})$ ($N+1\leq i\leq M$) will have some $\{e_j\}_{j=1}^{N}$ selected more than one time (together with the selection process for $1\leq i\leq N$).  Let $1\leq k\leq N$.  Define the index set $J_k\equiv \{i: j_i=k\}$ for all $N+1\leq i\leq M$.   We may now choose the value of the weights $v_i$ by, as an alternative to that seen in
the proof of Proposition \ref{prop_ConventionalFrame},
\be \label{v_i}
v_i^2=\left\{\begin{array}{ll}
\frac{1}{(|J_i|+1)\|\bar x_i\|^2}, & 1\leq i\leq N\\
\frac{1}{(|J_{j_i}|+1)\|\hat x_i\|^2}, & N+1\leq i\leq M, \ j_i\in \{1,\cdots,N\}
            \end{array}\right.
\ee
Then the new fusion frame operator
\[
\cSw=\sum_i v_i^2\cP_i^*\cP_i = Id.
\]

With such selections of non-orthogonal projections $\cP_i$ and the associated weights $v_i$, we have constructed a (non-orthogonal) Parseval fusion frame $\ds \left\{\cP_i, v_i^2\right\}$, where $v_i^2$
are as given in (\ref{v_i}).

\subsection*{The frame expansion via the Parseval fusion frame}

Let us now check what the corresponding frame expansion looks like in the previous example.  For this, we need to figure out the expression of $\cP_i$. We will use {\it pseudoframes for subspaces} (PFFS) \cite{LiOgPFFS98} as a tool.    Recall that PFFS is a frame-like expansion for a subspace $W$.  Specifically, let $W$ be a closed subspace of $\cH$.   Let $\{x_n\}\subseteq\cH$ and $\{\tilde x_n\}\subseteq\cH$ be two Bessel sequences (not necessarily in $W$).  We say $\{x_n\}$ and $\{\tilde x_n\}$ form a pair of PFFS for $W$ if
\begin{equation}\label{PFFS}
f =\sum_n\< f, x_n\> \tilde x_n\tab\text{for every }f\in {W}
\end{equation}

One important feature of PFFS is that (\ref{PFFS}) is the non-orthogonal projection $\cP_{W,\sp\{x_n\}^{\perp}}$ from $\cH$ onto $W$ along the direction $\sp\{x_n\}^{\perp}$ \cite{LiOgPFFS98}.

We also point out that if $\{w_n\}$ and $\{\tilde w_n\}$ are conventional frames of $W$, then, for any $\{z_n\}\subseteq W^{\perp}$,  the pseudoframe sequence $\{x_n\}$ is always given by \cite{LiOgPFFS98}, (see the details in Section \ref{secPFFS})
\[
    x_n=w_n + z_n.
\]
One especially useful implication of this characterization is that the direction of the projection $\sp\{x_n\}^{\perp}$ can be freely adjusted by choosing $\{x_n\}$ properly, which in turn can be accomplished by selecting an appropriate orthogonal sequence $\{z_n\}\subseteq W^{\perp}$.
%

Let us now show how PFFS is applied in this particular example.  Let $\{y_i\}$ be the corresponding PFFS sequence associated with the choice of $\cP_i$ in (\ref{Pi_1}) and (\ref{Pi_2}) (we will show how $\{y_i\}$ are constructed immediately later).  Then, for all $f\in \cH$ we have
\begin{eqnarray}
f &=& \cSw^{-1}\cSw f = \sum_i v_i^2\cP^*_i\cP_i f \nonumber \\
  &=& \sum_i v_i^2 \cP^*_i\left(\< f, y_i\> x_i\right) \nonumber \\
  &=& \sum_i v_i^2 \< f, y_i\>\< x_i,x_i\> y_i  \nonumber \\
  &=& \sum_i v_i^2 \|x_i\|^2 \< f, y_i\> y_i  \label{step1}
\end{eqnarray}

\subsection*{The determination of the sequence $\{y_i\}$}

Since $W_i=\sp\{x_i\}$, the canonical dual frame of $x_i$ in $W_i$ is $x_i/\|x_i\|^2$.  Hence
\[
  y_i =\frac{x_i}{\|x_i\|^2} +z_i,
\]
where $z_i\in W_i^{\perp} =\sp\{x_i\}_{i=i}^{\perp}$.  The choice of $z_i$ must also simultaneously satisfy (\ref{Pi_1}) or (\ref{Pi_2}) (depending on the value range of $i$).

That $z_i\in\sp\{x_i\}_{i=i}^{\perp}$ suggests $z_i=(z_{ik})_k$ must be in the co-dimension 1 subspace,
\be\label{z_i_in_Wi_perp}
x_{i1}z_{i1}+x_{i2}z_{i2}+\cdots +x_{iN}z_{iN} =0 .
\ee

Also,
(\ref{Pi_1}) and (\ref{Pi_2}) further requires that $\cN(\cP_i)=\sp\{y_i\}_{i=i}^{\perp} = \cN(e_i)$
 (for $1\leq i\leq N$) or $\cN(\cP_i) =\sp\{y_i\}_{i=i}^{\perp} = \cN(e_{j_i})$ (for $N+1\leq i\leq M$).
 These suggest that, when $1\leq i\leq N$,
\be \label{Pi_1_direction}
    y_{ik}=0, \tab \forall k\neq i,
\ee
and when $N+1\leq i\leq M$,
\be \label{Pi_2_direction}
    y_{ik}=0, \tab \forall k\neq j_i,
\ee
where $j_i\in \{1,\cdots,N\}$ is as seen in the previous discussion.
Equations (\ref{z_i_in_Wi_perp}), (\ref{Pi_1_direction}) and (\ref{Pi_2_direction}) give rise to
\[
z_{ik}=\left\{\begin{array}{ll}
 -\frac{x_{ik}}{\|x_i\|^2}, & k\neq i \\
  \frac{\sum_{j\neq i}x_{ij}^2}{x_{ii}\|x_i\|^2}, & k=i\end{array}\right. \ \ i=1, \cdots, N,
\]
and
\[
z_{ik}=\left\{\begin{array}{ll}
 -\frac{x_{ik}}{\|x_i\|^2}, & k\neq j_i \\
  \frac{\sum_{j\neq j_i}x_{ij}^2}{x_{ij_i}\|x_i\|^2}, & k=j_i,\end{array}\right. \ \ i=N+1,\cdots , M.
\]
We have therefore,
\[
y_i =\left(0,\cdots,0,\frac{1}{x_{ii}},0,\cdots,0  \right)^H, \ \ i=1, \cdots, N,
\]
where $\ds \frac{1}{x_{ii}}$ appears in the $i^{th}$ position, and $x^H$ stands for Hermitian transpose of $x$.  Likewise,
\[
y_i =\left(0,\cdots,0,\frac{1}{x_{ij_i}},0,\cdots,0  \right)^H, \ \ i=N+1, \cdots, M,
\]
where $\ds \frac{1}{x_{ij_i}}$ appears in the $j_i^{th}$ position.  These are certainly consistent with the matrix representation of $\cP_i$ as in, e.g., (\ref{Pi_dim1}).

We see that the PFFS sequence $\{y_i\}$ in this example of Parseval fusion frames of 1-dimensional subspaces is really a multiple of the orthonormal basis
 $\{e_i\}$ with possible repeats of some elements of $\{e_i\}$'s.

\begin{remark}
(1)    Let us now check
 the frame expansion and reconstruction with the given frame $\{x_n\}$.  Bringing back the non-orthogonal Parseval fusion frame expansion (\ref{step1}), with  $v_i^2$'s as given in (\ref{v_i}),
 we have for all $f\in \cH$,
\begin{eqnarray}
f &=& \sum_i v_i^2\|x_i\|^2 \<f, y_i\> y_i \nonumber \\
  &=& \sum_i v_i^2 \|x_i\|^2 \left\< f, \frac{x_i}{\|x_i\|^2} + z_n\right\> y_i \nonumber \\
  &=& \sum_i v_i^2 \< f, x_i\> y_i.   \label{step2}
\end{eqnarray}
We see that there is no need to calculate dual frames!  With coefficients expanded by the frame sequence $\{x_n\}$, a (pseudo) dual sequence for the reconstruction is the $\{y_i\}$ constructed above.  The only computation is about $v_i^2$ and one non-zero reciprocal in each $y_i$.  Besides the proper weighting factor, this expansion becomes a linear combinations of weighted (and possible repeated) orthonormal basis elements $\{e_i\}$.

(2)   Suppose this is for (sensor) signal measurement applications.  It is important to point out that the PFFS implementation still makes use of the given analysis coefficients $\{\< f, x_i\>\}$ (or sensor measurements in practical applications).
In particular, if $f\in \sp\{x_n\}$, $\{\< f, y_i\>\}=\{\< f, x_i\> \}$.  This is also a critical point in this non-orthogonal fusion frame - that the non-orthogonal projection does not require additional measurements in case the signal measured is contained in the sensor subspace.

(3)   In light of (\ref{step1}), given any frame $\{x_n\}$, the frame expansion (\ref{step2}) we constructed here is also equivalent to a related Parseval frame expansion consisting of weighted and repeated
orthonormal basis elements.   To see this, we bring back (\ref{step1}) and (\ref{step2})
and compute for all $f\in \cH$,
\begin{eqnarray*}
f &=& \sum_i v_i^2 \< f, x_i\> y_i  \\
  &=& \sum_i v_i^2\|x_i\|^2\< f, y_i\> y_i \\
  &=& \sum_i \< f, v_i\|x_i\|y_i\> v_i\|x_i\|y_i,
\end{eqnarray*}
which is a Parseval frame expansion whose elements $\{u_i\}$ are weighted and possibly repeated
orthonormal basis elements, namely $\ds u_i=\frac{v_i\|x_i\|}{x_{ii}}e_i$ for $1\leq i\leq N$ and
$\ds u_i=\frac{v_i\|x_i\|}{x_{ij_i}}e_{j_i}$ for $N+1\leq i\leq M$.   If we bring
 the choice of $v_i^2$ of (\ref{v_i}) back into the picture, it becomes clear that
\[
u_i=\left\{\begin{array}{ll}
\frac{1}{\sqrt{|J_i|+1}}e_i, & 1\le i\le N \\
\frac{1}{\sqrt{|J_{j_i}|+1}}e_{j_i}, & N+1\le i\le M, \ j_i\in \{1,\cdots,N\}.
           \end{array}\right.
\]
Consequently, the frame expansion (\ref{step2}) is equivalent to a Parseval frame expansion with element $u_i$'s being orthonormal basis elements, repeated and thereby correctly scaled.

Note that there are many different such Parseval frames $\{u_i\}$, with different repetition factors, judging from the derivation process.

\end{remark}


\section{Diagonal characterization of $\cP_W^*\cP_W$} \label{sec_diag}

In a simplified version of Problem \ref{prob1}, we consider in this section the conditions for which one individual $\cP_W^*\cP_W$ can be diagonal.

\begin{proposition}\label{2}
Fix an orthonormal basis $\{e_i\}_{i=1}^N$ for $\mathcal{H}$.  Let $W\subseteq\mathcal{H}$ be a $k-$dimensional subspace.  The following are equivalent:
\begin{itemize}
\item[(i)] There exists a projection $\mathcal{P}_{W}$ such that the matrix of $\mathcal{P}_W^*\mathcal{P}_W$ is diagonal with respect to $\{e_i\}_{i=1}^N$.
\item[(ii)] There exists a subset $K\subseteq\{1,2,...,N\}$, $|K|=k$ such that there exists an orthogonal basis $\{x_i\}_{i\in K}$ for $W$ such that $x_i(j)=\delta_{ij}$ for $i,j\in K$.
\item[(iii)] There exists a projection  $\mathcal{Q}_W$ such that $\{\mathcal{Q}_We_i\}_{i\in K}$ is an orthogonal basis for $W$ and $\mathcal{Q}_We_i(j)=\delta_{ij}$ for $i,j\in K$.
\item[(iv)] There exists a projection $\mathcal{P}_W$ such that $\{\mathcal{P}_We_i\}_{i\in K}$ is an orthogonal basis for $W$.
\item[(v)] There exists a subset $K\subset \{1,2,\ldots,N\}$, $|K|=k$ and there exists an orthogonal
basis $\{x_i\}_{i\in K}$ for $W$ such that $\pi_Kx_i$ is an orthonormal basis for span $\{e_i\}_{i\in K}$.
\end{itemize}
Moreover, in all of the above cases, the diagonal elements of $\PP_W^*\PP_W$ are
$\|x_i\|^2$ for cases (ii) and (v);  $\|\QQ_We_i\|^2$ in case (ii), and $\|\PP_We_i\|^2$ for case (iv).
\end{proposition}
\begin{proof}
$(iii)\Rightarrow(iv)\Rightarrow(i)$ is clear.

We first show $(i)\Rightarrow(ii)$.  We know that $\langle\mathcal{P}_We_i,\mathcal{P}_We_j\rangle=0$ for $i\neq j$, so $\{\mathcal{P}_We_i\}_{i=1}^N$ is an orthogonal set.  But $\dim(W)=k$ so there exists a $K\subseteq\{1,...,N\}$ such that $\mathcal{P}_We_i=0$ for $i\not\in K$ and $\{\mathcal{P}_We_i\}_{i\in K}$ is an orthogonal basis for $W$.  Let $V=\text{span}\{e_i\}_{i\in K}$.  Observe that for $x,y\in V$
\begin{equation}\label{1}
\langle x,y\rangle=0 \mbox{ if and only if } \langle \mathcal{P}_Wx,\mathcal{P}_Wy\rangle=0.
\end{equation}
Now write
$$
\mathcal{P}_We_i=\pi_V\mathcal{P}_We_i+(I-\pi_V)\mathcal{P}_We_i
$$
to see that
$$
\mathcal{P}_W\mathcal{P}_We_i=\mathcal{P}_W\pi_V\mathcal{P}_We_i+\mathcal{P}_W(I-\pi_V)\mathcal{P}_We_i.
$$
But $\mathcal{P}_W(I-\pi_V)\mathcal{P}_We_i=0$, since $\mathcal{P}_We_i=0$ for $i\not\in K$.  Therefore, since $\mathcal{P}_W$ is a projection (\textit{i.e.}, $\mathcal{P}_W^2=\mathcal{P}_W$) we have that
$$
\mathcal{P}_We_i=\mathcal{P}_W\pi_V\mathcal{P}_We_i.
$$
Hence, (\ref{1}) now implies that $\{\pi_V\mathcal{P}_We_i\}_{i\in K}$ is an orthogonal basis for $V$.  Now observe that $\pi_V$ is a bijection between $V$ and $W$ so we can choose $\{x_i\}_{i\in K}$ so that $\pi_Vx_i=e_i$.

$(ii)\Rightarrow(iii)$.  By $(ii)$ we know that there is an orthogonal basis $\{x_i\}_{i\in K}$ for $W$ with the desired properties, so we just need to show that there is a projection $\mathcal{Q}_W$ such that $\mathcal{Q}_We_i=x_i$ for $i\in K$.  Define $U:\mathcal{H}\rightarrow\mathcal{H}$ by
$$
Ue_j = \begin{cases} 0, & \mbox{if } j\not\in K \\ x_j, & \mbox{if } j\in K. \end{cases}
$$
We now claim that $\mathcal{Q}_W=U\pi_V$ satisfies $(iii)$.  Clearly, $\mathcal{Q}_We_i=x_i$, so we just need to check that $\mathcal{Q}_W$ is in fact a projection.  If $j\not\in K$ then clearly $\mathcal{Q}_W^2e_j=0$.  If $j\in K$, then $\mathcal{Q}_{W}^2e_j=\mathcal{Q}_{W}x_j=U\pi_Kx_j=Ue_j=x_j$ so $\mathcal{Q}_W$ is a projection.

$(ii) \Rightarrow (v)$ is obvious.

$(v) \Rightarrow (i)$:  Define,
\[ \PP_We_j = 0 \ \ \mbox{if $j\notin K$},\ \ \PP_W\pi_K(x_j) = x_j.\]
It is immediate that $\PP_W$ is a projection.  Also, $\PP_W$ is an orthogonal operator
when restricted to span $\{e_i\}_{i\in K}$.  Hence, if $i,j\in K,\ i\not= j$ we have
\[ \langle \PP_W^*\PP_We_i,e_j\rangle = \langle \PP_We_i,\PP_We_j\rangle =0.\]
On the other hand, if $j\notin K$ then $\PP_We_j = 0$ and so
\[ \langle \PP_W^*\PP_We_i,e_j\rangle = \langle \PP_We_i,\PP_We_j\rangle
= \langle \PP_We_i,0\rangle =0.\]
So we have $(i)$.

The ``moreover" part of the theorem is obvious from the proofs.
\end{proof}

We now check that for large dimensional subspaces $W$ of $\mathcal{H}$, there is a fundamental restriction for finding a projection $\mathcal{P}_W$ onto $W$ so that $\mathcal{P}_W^*\mathcal{P}_W$ is diagonal with respect to $\{e_i\}$.

\begin{proposition}
Let $\mathcal{H}$ be an $N$ dimensional Hilbert space and $W$ a $k$-dimensional subspace.  If $k>\frac{N}{2}$ and there is a projection $\mathcal{P}_W$ such that the matrix of $\mathcal{P}_W^*\mathcal{P}_W$ is diagonal with respect to $\{e_i\}_{i=1}^N$ then there are at least $2k-n$ $e_i$'s in $W$.
\end{proposition}
\begin{proof}
By proposition \ref{2} part (ii) we can find a $K\subseteq\{1,...,N\}$ such that $|K|=k$ and an orthogonal basis $\{x_i\}_{i\in K}$ for $W$ which satisfies $x_i(j)=\delta_{ij}$ for $i,j\in K$.  Therefore, we know that $\langle\pi_Vx_i,\pi_Vx_j\rangle=0$ for $i,j\in K$ which means that $\langle(I-\pi_V)x_i,(I-\pi_V)x_j\rangle=0$ for $i,j\in K$ (since we know $\langle x_i,x_j\rangle=0$).  Therefore $\{(I-\pi_V)x_i\}_{i\in K}$ is an orthogonal set inside an $N-k$ dimensional space, which means there is a $J\subseteq K$ such that $|J|\geq 2k-N$ and $(I-\pi_V)x_j=0$ for every $j\in J$.  Then for each $j\in J$ we actually have $x_j=e_j$.
\end{proof}


%

\section{The implementation of $\cPw$ via pseudoframes for subspaces}\label{secPFFS}

Pseudoframes for subspaces (PFFS) \cite{LiOgPFFS98} are an extension of frames for subspaces $\cW$ where both frame-like sequences $\{x_n\}$ and $\{\tilde x_n\}$ are not necessarily in $\cW$, yet for every $f\in W$
\[
f=\sum_n\< f, x_n\> \tilde x_n.
\]
Furthermore, the frame-like condition holds for all vectors in $\cW$.  Namely, there are constants $0<A\leq B<\infty$ such that for all $f\in W$
\[
A\|f\|^2 \leq\sum_n |\< f, x_n\> |^2 \leq B\|f\|^2.
\]
Bringing in a projection operator $\cP$ onto $\cW$, and a PFFS gives rise to
\begin{equation}\label{pffs}
\cP g=\sum_n\<\cP g, x_n\> \tilde x_n.
\end{equation}
for every $g\in\mathcal{H}$.

\subsection{$\cP$-consistent PFFS and non-orthogonal projections}

We recall the property of $\cP$-consistent PFFS with an assumption that the sequence $\{x_n\}$ is Bessel in $\cH$.  Using the same terminology of Aldroubi and Unser \cite{UnserAldroubi94}, we say a PFFS is $\cP$-consistent \cite{LiOg_noiseSuppression} if
$U\cP=U$, where $U:\cH\rightarrow l^2$ is the analysis operator functioning as the measuring device in the form $U f =\{\< f, x_n\>\}$ for all $f\in\cH$.  The $\cP$-consistent principle is to say that the direct measurement of a function $f$ equals the measurement of a projection (approximation) $\cP f$ of $f$ onto (in) $W$.   This clearly depends on the direction of the projection.  We also recall that a $\cP$-consistent PFFS expansion is precisely a non-orthogonal projection operator, and the direction of the projection can be arbitrarily adjusted by steering the $\sp\{x_n\}$ \cite{LiOgPFFS98}, \cite{LiOg_noiseSuppression}.

It is known that a PFFS is a $\cP$-consistent PFFS if and only if the direction of the projection (or the null space of $\cP$) $\ds\cN(\cP)=\sp\{x_n\}^{\perp}$ \cite{LiOg_noiseSuppression}, and it is always achievable.
Consequently, one can always have $\cN(\cP)=\sp\{x_n\}^{\perp}$, and the range of $\cP^*$, $\cR(\cP^*)=\cN(\cP)^{\perp}=\sp\{x_n\}$.  The resulting non-orthogonal projection is given by
\[
\cP g  = \sum_n\< g,\cP^* x_n\> \tilde x_n = \sum_n\< g, x_n\> \tilde x_n,
\]
for every $g\in\mathcal{H}$.

If we denote by $\cP_{W,\, \cN(\cP)}$ the projection operator with the first index $W$ as the range, and the second index $\cN(\cP)$  as the ``direction'' of the projection, then PFFS always produces a projection onto $W$ along the direction $\sp\{x_n\}^{\perp}$, namely $\cP_{W,\, \sp\{x_n\}^{\perp}}$.

From another point of view, if $\cN(\cP)$ is given, one can always select $\{x_n\}$ so that $\sp\{x_n\}=\cN(\cP)^{\perp}$, and thereby construct a non-orthogonal projection via PFFS.  More importantly, the selection of $\{x_n\}$ for a given $W$ is made easy by the following proposition.

\begin{theorem} \label{thm_PFFS_vs_frame_for_X} {\rm\cite{LiOgPFFS98}}
Let $\{x_n\}$ be a Bessel sequence with respect to $W$, and let $\{\tilde x_n\}$ be a
Bessel sequence in $\cH$. The following are equivalent:
\begin{enumerate}
\itemsep -2pt
\item $\{x_n, \tilde x_n\}$ is a PFFS for $W$.
\item If $\pi_W$ is the orthogonal projection of $\cH$ onto $W$, both of
the following hold:
\begin{enumerate}
    \itemsep -2pt
    \item $\{\pi_W x_n\}$ is a frame for $W$ with a dual frame $\{\pi_W \tilde x_n\}$.
    \item For all $f\in W$ we have $\ds\sum_n\< f, \pi_W x_n\> (I-\pi_W)\tilde x_n = 0$.
\end{enumerate}
\item There is a frame $\{w_n\}$ of $W$ with a dual frame $\{\tilde w_n\}\subseteq W$, a
sequence $\{z_n\}$ in $(I-\pi_W)\cH$ and a sequence $\{y_n\}\in {\mathcal
U}(\{w_n\})$ and a unitary operator $T: \sp\{y_n\}\rightarrow (I-\pi_W)\cH$ so
that
\[
   x_n=w_n+z_n
\]
and
\[
   \tilde x_n =\tilde w_n +T(y_n).
\]
\end{enumerate}
\end{theorem}
Here ${\mathcal U}$ is defined as follows.  If $\{w_{n}\}$ is a frame for $W$, then
\[
{\mathcal U} (\{x_n \}) \equiv\left\{\mbox{Bessel} \{y_n \}:
\sum_{n} \langle f, x_n \rangle y_n = 0,\
\mbox{for all $f\in {W}$}\right\}.
\]

Therefore, if we construct $\{x_n\}$ so that $\sp\{x_n\}=\cN(\cP)^{\perp}$, it turns out the choice of $\{x_n\}$ is fairly easy - adding to a frame $\{w_n\}$ of $W$ orthogonal components $\{z_n\}\subseteq W^{\perp}$ so that $\sp\{x_n\} = \cN(\cP)^{\perp}$.

\subsection{Sensor measurements were not altered while implementing a $\cP_{\cW}$ via PFFS}

Suppose fusion frames are applied in sensor network data collection applications.   Each sensor is then spanned by a sensory frame $\{w_n\}$ given by the elementary transformation (often simple shifts) of the spatial reversal of the sensor's impulse response function \cite{LiYan08}.

The measurement of each sensor is thus given by $\{\< f, w_n\>\}$ a-priori by the physics of the sensor.   Any post processing/fusion operation would have to make use of such a-priori measurements.  Implementation through PFFS can achieve that.  Recall, for all $f\in\mathcal{H}$
\[
\cP_{\cW} f=\sum_n\< f, x_n\> \tilde x_n = \sum_n\< f, w_n + z_n\> \tilde x_n =\sum_n \left(\< f, w_n\> +\< f, z_n\>\right)\tilde x_n.
\]
Consequently, we will just need to add a controlled measurement to the sensor's
complement subspace via $\{\< f,z_n\>\}$.  The implementation of any non-orthogonal projection $\cP_{\cW}$ will be achieved together with the a-priori sensor measurement $\{\< f, w_n\>\}$.   In particular, if a signal $f$ is within the sensory subspace spanned by $\sp\{w_n\}$, then $\< f, x_n\> =\< f, w_n\>$.

\subsubsection*{Subspace transformation will not do} \ \
It is worth mentioning that techniques of subspace rotation or transformation with the purpose of diagonalizing orthogonal projection operators would not be able to make use of the sensor measurements $\{\< f, w_n\>\}$.  This is because diagonalization of $\pi_{\cW}$ involves a unitary operator $T$ such that
\[
D = T^H\left({\tilde X}^H X\right) T,
\]
where $X$ is the matrix with frame elements $\{w_n\}$ as its columns.

On the one hand, it seems that a transformation in the form of $F =XT$, which is equivalent to the rotation of the subspace, would have had the orthogonal projection $\pi_{\cW}$ diagonalized.   On the other hand, the new frame system $F=XT$ would have to ``measure'' functions $f$ by $XT(f)$.  But the $T(f)$ part simply does not exist  in (at least) sensor network and distributed processing applications.

This is why non-orthogonal fusion frames is a much more natural tool to achieve the sparsity of the fusion frame operator.

\subsection{The matrix representation of the (new) fusion frame operator $\cSw$}

For computational needs, we show that the new fusion frame operator has a natural matrix representation via PFFS.

Let $\{w_n\}$ and $\{\tilde w_n\}$ be a frame and a dual frame, respectively, of the subspace $W$.  Form $x_n = w_n + z_n$ with $z_n\in W^{\perp}$ for all $n$ such that $\sp\{x_n\}=\cN(\cP)^{\perp}$, and
note that $\{x_n\}$ is a Bessel sequence of $\cH$.
Then for every $f\in\mathcal{H}$
\[
\cP_{W,\sp\{x_n\}^{\perp}} f = \sum_n\< f, x_n\> \tilde w_n.
\]
Consequently, if $X$ is the matrix with $\{x_n\}$ as columns, and $Y$ is the matrix with $\{\tilde w_n\}$ as columns, then a natural matrix representation of $\cP_{W,\sp\{x_n\}^{\perp}}$ is
\[
    \cP_{W,\sp\{x_n\}^{\perp}} = YX^{H}.
\]
As a result, the fusion frame operator $\cSw$ is represented by
\[
\cSw= \cP^* \cP = XY^{H}YX^{H}.
\]

\section{Eigen-properties of $\cP_W$}

We will compare non-orthogonal projections to orthogonal projections.
The first result is an alternative construction tool.

\begin{proposition}\label{prop1}
Let $W$ be a $k$-dimensional subspace of $\cH_N$ and $\PP_W$ be
a (non-orthogonal) projection onto $W$.  Let:

(1)  $\{x_i\}_{i=1}^k$ be an orthonormal basis for $W$.

((2)  $\{y_i\}_{i=k+1}^N$ be an orthonormal basis for the $N-k$-dimensional
space $V=(I-\PP_W)\cH_N$.

Then:
\[ \{x_i,y_j\}_{i=1,\ j=k+1}^{\  k\ \  \ \ N}\]
are the eigenvectors of $\PP_W$ with eigenvalues ``1" for $x_i$ and
eigenvalues "0" for $y_j$.

In particular, $\PP_W$ is an orthogonal projection if and only if $V=W^{\perp}$.

\end{proposition}

\begin{proof}
Since $\PP_W$ is a projection, we have
\[ \PP_Wx_i = x_i.\]
i.e.  $x_i$ is an eigenvector for $\PP_W$ with eigenvalue ``1".  Also,
\[ \PP_Wy_i = 0, \]
So $y_i$ is an eigenvector for $\PP_W$ with eigenvalue ``0".
\end{proof}

The converse of the above proposition is also true.

\begin{proposition}\label{prop2}
Given $W,V$ subspaces of $\cH_N$ with $W\cap V = \{0\}$, and
\[ dim \ W=k,\ \ \mbox{and dim}\ V = N-k.\]
Choose orthonormal bases $\{x_i\}_{i=1}^k$ and $\{y_i\}_{i=k+1}^N$
for $W$ and $V$ respectively.  Given $x\in \cH_N$, there are unique scalars
$\{a_i\}_{i=1}^N$ so that
\[ x = \sum_{i=1}^k a_ix_i + \sum_{i=k+1}^N a_iy_i.\]
Define
\[ \PP_W(x) = \sum_{i=1}^k a_ix_i.\]
Then $\PP_W$ is a projection on $\cH_N$ (and hence, $\PP_W$ has eigenvectors
$\{x_i,y_j\}_{i=1,j=k+1}^{ \ k\ \ ,\ \ N}$ and eigenvalues ``1" for $x_i$ and
``0" for $y_j$).
\end{proposition}

\begin{proof}
We compute:
\[ \PP_W(\PP_Wx) = \PP_W(\sum_{i=1}^k a_i x_i) = \sum_{i=1}^k a_ix_i= \PP_W(x)\ \ \mbox{(by
definition)}.\]
\end{proof}

The above tells us what we can get out of non-orthogonal
projections if we are projecting along a subset of the basis.  To keep the notation
simple, we will project {\em along} span $\{e_i\}_{i=k+1}^N$.  But this clearly
works exactly the same for any $K\subset \{1,2,\ldots,N\}$ with $|K|=k$.

\begin{corollary}\label{cor1}
In $\cH_N$, let $K=\{1,2,\ldots, k\}$.  Choose an orthogonal set of vectors
$\{y_i\}_{i=1}^k$ in $(I-\PP_K)\cH_N = \PP_{K^c}\cH_N$ and for each
$i=1,2, \ldots ,k$ let
 \[ x_i =\frac{1}{\|e_i+y_i\|} e_i+\frac{1}{\|e_i+y_i\|}y_i.\]
   (Note that if $N-k < k$,
then $2k-N$ of the $y_i$ will be zero).  Let $W$ = span $\{x_i\}_{i=1}^k$.
Define $\PP_W:\cH_N \rightarrow W$ by
\[ \PP_We_i = \|e_i+y_i\|x_i, \ \ \mbox{if}\ \ i\in K,\]
and
\[ \PP_We_i =0,\ \ \mbox{if}\ \ i\in K^c.\]
Then $\PP_W$ is a (non-orthogonal) projection having eigenvectors $\{x_i\}_{i=1}^k$
with eigenvalues  ``1" and eigenvectors $\{e_n\}_{n=N-k}^N$  with eigenvalues ``0" for $n\in K^c$.

Moreover, $\PP_W^{*}\PP_W$ is a diagonal matrix with eigenvectors $\{e_n\}_{n=1}^N$
and eigenvalues ``0" for $i=k+1,k+2,\ldots, N$ and eigenvalues
$\|e_i+y_i\|^2=1+\|y_i\|^2$ for $i=1,2,\ldots, k$.

Finally, if $Q$ is the orthogonal projection of $\cH_N$ onto the same span $W$, then
$Q$ has eigenvectors $\{x_i\}_{i=1}^k$ with eigenvalues ``1" and eigenvectors
$\{z_i\}_{i=k+1}^N$ an orthonormal basis for $W^{\perp}$ with eigenvalues ``0".
\end{corollary}

\begin{proof}
For $x_i \in W$,
\begin{eqnarray*}
\PP_W(x_i)&=& \PP_W( \frac{1}{\|e_i+y_i\|}e_i )+ \PP_W(\frac{1}{\|e_i+y_i\|}y_i)\\
&=& \frac{1}{\|e_i+y_i\|}\PP_We_i + 0\\
&=& \|e_i+y_i\|(\frac{1}{\|e_i+y_i\|} )x_i\\
&=& x_i.
\end{eqnarray*}
So $\PP_W$ is a projection.

For the {\em moreover part}, if  $j$ is not in $K$ then
\[ \langle \PP_W^{*}\PP_We_i,e_j \rangle = \langle \PP_We_i,\PP_We_j\rangle
= \langle \PP_We_i,0\rangle =0.\]
If $i\not= j\in K$ then
\begin{eqnarray*}
\langle \PP_W^{*}\PP_We_i,e_j \rangle &=& \langle \PP_We_i,\PP_We_j\rangle\\
&=&\langle \|e_i+y_i\| x_i,\|e_j+y_j\|x_j\rangle\\
&=& \|e_i+y_i||\|e_j+y_j\|\langle x_i,x_j\rangle = 0.
\end{eqnarray*}
And if $i=j\in K$ then
\[ \langle \PP_W^{*}\PP_We_i,e_i \rangle = \|\PP_We_i\|^2 = \|e_i+y_i\|^2.\]

The {\em finally part} is clear.
\end{proof}

\noindent {\bf Remark}:  It is worth understanding  intuitively why $\PP_W^{*}\PP_W$
has all of its non-zero eigenvalues $\ge 1$.  This is happening because by forcing
ourselves to project along a set of $e_j's$, we see that $\PP_W$ must project a
set of vectors of the form $e_i$ to vectors of the form $e_i+y_i$ where
$y_i \perp e_i$, and hence
\[ \|\PP_We_i \|^2= \|e_i\|^2+\|y_i\|^2 \ge \|e_i\|^2=1.\]

Now we can see what diagonal entries we can get when $\PP_W^{*}\PP_W$ is
a diagonal matrix.

\begin{corollary}\label{cor2}
Fix $1\le k\le N$ and choose $K\subset \{1,2,\ldots,N\}$ with $|K|=k$.

(1)  If dim $k \le \frac{N}{2}$ and any numbers $\{a_n\}_{n\in K}$ are given
with $a_n \ge 1$, there is a subspace $W$ of $\cH_N$ and a (non-orthogonal) projection
$\PP_W$ onto $W$ so that the eigenvectors of $\PP_W^{*}\PP_W$ are
$\{e_n\}_{n=1}^N$  with respective eigenvalues $\{a_n\}_{n\in K}$ and ``0"
if $n\notin K$.  That is, $\PP_W^{*}\PP_W$ is a diagonal matrix with non-zero
diagonal entries $\{a_n\}_{n\in K}$.

(2)  If $k> \frac{N}{2}$, there is a $K_1 \subset K$ with $|K_1| = N-k$
and if $\{a_n\}_{n\in K_1}$ are given with $a_n\ge 1$, then there is a
subspace $W$ of $\cH_N$ and a  (non-orthogonal)
projection $\PP_W$ onto $W$ so that $\PP_W^{*}\PP_W$ is a diagonal matrix with diagonal
entries ``0" if $n\notin K$, diagonal entries ``1" if $n\in K\setminus K_1$,
and diagonal entries $a_n$ if $n\in K_1$.
\end{corollary}

\begin{proof}(1)
Since $dim\ W \le \frac{N}{2}$, we have that
\[ N - \frac{N}{2}= \frac{N}{2} \ge dim\  W.\]
Hence, there is an orthogonal set of vectors $\{y_n\}_{n\in K}$ satisfying:
\vskip12pt
(a)    $y_n \in \PP_{K^c}\cH_N$.
\vskip12pt
(b)  $\|y_n\|^2 = a_n-1$.
\vskip12pt
By Corollary \ref{cor1}, there exists a subspace $W$ with
\[ W = span \{\frac{1}{\|e_n+y_n\|}\left ( e_n+y_n \right):n\in K\},\]
and a projection $\PP_W$
so that $\PP_W^{*}\PP_W$ has eigenvectors $\{e_n\}_{n=1}^N$ and non-zero eigenvalues
only for $n\in K$ which are of the form:
\[ \|e_n+y_n\|^2 = \|e_n\|^2 + \|y_n\|^2 = 1+ (a_n-1) = a_n.\]
\vskip12pt
(2) We just do as in (1) except now, we can only find $N-k$ orthogonal vectors
$y_n$ in $\PP_{K^c}\cH_N$.  So we pair these $y_n's$ with $N-k$ of the $e_n's$
in $K$ and put $e_n\in W$ for the rest of the $n\in K$.
\end{proof}

\section{Tight and multiple fusion frames}

Nonorthogonal fusion frames bring in some
quite unique
 properties that the orthogonal fusion frames do not have.   For instance, we can now easily construct examples of tight fusion frames for non-orthogonal projections which do not exist in orthogonal fusion frames.  In fact, we may have
quite spectacular examples where tight fusion frames can be constructed via one (proper) subspace.

One immediate observation is that we may have multiple non-orthogonal projections onto one given subspace, now that (non-orthogonal) projections are no longer unique.  We show that by applying multiple projections onto one and each subspace, tight nonorthogonal fusion frames exist.

\begin{remark}  First, let us observe that there is an obvious restriction on the number of projections we need.  That is, if $W$ has dimension $k$ in $\cH_N$, then each projection onto $W$ has at most $k$ non-zero eigenvalues (and $P_W^*P_W$ also
has the same).  So if we want $\sum_{i=1}^{L} \PP_i^{*}\PP_i = \lambda I$, then
\[ L \ge \lfloor \frac{N}{k}\rfloor +1.\]

In the next proposition we will see that this works if
$k$ divides $N$.  However, it can be shown that if $k$ does not divide $N$
then this result fails in general.
\end{remark}

\begin{proposition}\label{prop3}
Let $W\subseteq\mathcal{H}$ be a subspace of dimension
$k\ge 1$.

(1)  If $k\geq \frac{N}{2}$, there are
non-orthogonal projections $\{\PP_i\}_{i=1}^{2}$ onto $W$ so that
\[ \PP_1^{*}\PP_1 + \PP_{2}^*\PP_2 = 2I.\]

(2)  If $N=kL$, there are
non-orthogonal projections $\{\PP_i\}_{i=1}^{L}$ onto $W$ so that
\[ \sum_{i=1}^{L}\PP_i^{*}\PP_i = LI.\]

%
(3)  If $N=kL+M$ and $1\le M<k$, there exists a subspace $W$
of dimension $k$ and
non-orthogonal projections $\{\PP_i\}_{i=1}^{L+1}$ onto $W$ so that
$\sum_{i=1}^{L}\PP_i^{*}\PP_i$ has $\{e_i\}_{i=1}^N$ as eigenvectors
with eigenvalues ``L+1" for $\{e_i\}_{i=1}^{N-M}$ and eigenvalues ``L"
for $\{e_i\}_{i=N-M+1}^N$.
\end{proposition}

\begin{remark}
The problem with part (3) of the Proposition is that we can't move this back
to an arbitrary subspace $W$ since we are getting $\sum_{i\in I}\PP_i^*\PP_i$ a diagonal operator
here and
\[ U^*\left ( \sum_{i\in I} \PP_i^* \PP_i \right ) U= \sum_{i\in I}U^*\PP_i^*\PP_iU,\]
is not a diagonal operator.
\end{remark}

Before we prove the proposition, we give some simple examples to
show how the proof will work.
\vskip12pt
\noindent {\bf Example 1}:  There is a subspace $W$ in $\cH_3$
with dim $W$ =2 and two (non-orthogonal) projections $\PP_W$ and $\QQ_W$
giving a 2-tight fusion frame for $\cH_3$.
We also know \cite{CFMWZ10} that there is no tight fusion frame
for $\cH_3$ made from orthogonal projections and two, $2$-dimensional
subspaces.   Moreover,
the example above is "unique" in that the only way to produce
a 2-tight (non-orthogonal) fusion frame out of projections $P$
with $\PP^{*}\PP$ diagonal is to produce the above example up to applying a unitary operator.
\vskip12pt

To do this, we consider the 2-dimensional subspace of $\cH_3$ given by:
\[ W_1 = span\ \{(1,0,0),(0,\frac{1}{\sqrt{2}},\frac{1}{\sqrt{2}})\}.\]
Now, by our Corollary \ref{cor1}, if we project onto $W$ along $e_3$ with
$\PP_W$, then $\PP_W^{*}\PP_W$ will have eigenvectors $\{e_n\}_{n=1}^3$
with respective eigenvalues $\{1,2,0\}$ and if we project onto the subspace
$W$ along $e_2$ with $\QQ_W$, then $\QQ_W^{*}\QQ_W$ has eigenvectors
$\{e_n\}_{n=1}^3$ with respective eigenvalues $\{1,0,2\}$.  So
\[ \PP_W^{*}\PP_W + \QQ_W^{*}\QQ_W = 2I.\]
This example is unique since if we pick any two subspaces $W_1,W_2$ of
$\cH_3$ with dim $W_i$= 2 and choose any projections $\PP_{W_1},\PP_{W_2}$,
to get diagonal operators $\PP_{W_i}^{*}\PP_{W_i}$, then each projection must have
a unit vector in its span and be projecting along another unit vector.  Hence,
all you can get for eigenvalues is $\{1,a_1,0\}$ and $\{1,0,a_2\}$ and
we know that $a_1,a_2\ge 1$ and these two sets of 3 eigenvalues must be
arranged so their respective sums are all the same.  Checking cases we can
easily see that this only happens if the eigenvalues are lined up as they are
and hence $a_1=a_2 = 2$ and backtracking through the Corollary, we get
exactly our example back.
\vskip12pt

We give one more example which illustrates how the general case will go.
\vskip12pt

\noindent {\bf Example 2}:  In $\cH_{2N}$ let
\[ W = span\ \ \{e_i+e_{N+i}: 1\le i \le N\}.\]
Then dim $W$ = N (i.e. half the dimension of the space).
We will construct two projections $\PP_1,\PP_2$ so that
\[ \PP_1^{*}\PP_1 + \PP_2^{*}\PP_2 = 2I.\]
To do this let $\PP_1$ be the projection onto $W$ along $\{e_1,e_2,\ldots,e_N\}$
and $\PP_2$ the projection along $\{e_{N+1},e_{N+2},\ldots,e_{2N}\}$.  Then
by Corollary \ref{cor1} we have $\PP_{i}^{*}\PP_i$ has eigenvectors
$\{e_i\}_{i=1}^{2N}$ with non-zero eigenvalues $2$ for each and on
$\{e_{N+i}\}_{i=1}^N$ for $\PP_1^{*}\PP_1$ and $\{e_{i}\}_{i=1}^N$ for $\PP_2^{*}\PP_2$.
i.e.  $\PP_1^{*}\PP_1+\PP_2^{*}\PP_2 = 2I$.
\vskip12pt

For the proof of the Proposition we need a simple lemma.

\begin{lemma}\label{lem1}
Let $\{W_i\}_{i\in I}$ be subspaces of a Hilbert space $\cH$.
Assume there exists a unitary operator $U$
on $\cH$ and projections $\{\PP_{W_i}\}_{i\in I}$ onto the spaces $UW_i$
so that $\sum_{i\in I}\PP_{W_i}^{*}\PP_{W_i} = \lambda I$. Then
$\{U^{*}\PP_{W_i}U\}_{i\in I}$ is a family of projections onto $\{W_i\}_{i\in I}$
satisfying $\sum_{i\in I}U^{*}\PP_{W_i}^{*}\PP_{W_i}U = \lambda I$.
\end{lemma}

\begin{proof}
Since $U$ is unitary, $U^{*}$ is a unitary operator taking $UW_i$ onto $W_i$.
Also,
\[ U^{*}\PP_{W_i}U(U^{*}\PP_{W_i}U) = U^{*}\PP_{W_i}(UU^{*}\PP_{W_i}U
= U^{*}\PP_W\PP_WU = U^{*}\PP_{W_i}U.\]
That is, $U^*\PP_{W_i}U$ is a projection onto $W_i$.
Finally,
\[
\sum_{i\in I}U^{*}\PP_{W_i}^{*}\PP_{W_i}U = U^{*}\left ( \sum_{i\in I}\PP_{W_i}^{*}\PP_{W_i}\right )
= U^{*}\lambda IU=  U^{*}U\lambda I= \lambda I.\]
\end{proof}

\noindent {\bf Proof of Proposition \ref{prop3}}:
Let dim $W =k$.  We have to look at the three cases.
\vskip12pt
\noindent {\bf Case 1}:  We have $k \geq \frac{N}{2}$.
\vskip12pt
By Lemma \ref{lem1}, we may assume that our fixed subspace $W$ is
\[ W = span\ \ [ \{e_i+e_{k+i}\}_{i=1}^{N-k} \cup \{e_i\}_{i=N-k+1}^k] \]
By Corollary \ref{cor1}, if we project with $\PP_1$ onto $W$ along $\{e_i:i=1,2,\ldots,N-k\}$
we have that $\PP_1^*\PP_1$ has eigenvectors $\{e_i\}_{i=1}^N$ with respective
eigenvalues ``1" for $\{e_i\}_{i=N-k+1}^k$, ``2" for $i=1,2,\ldots,N-k$, and ``0"
otherwise.  Let $\PP_2$ be the projection along $\{e_i:i= k+1,k+2,\ldots,N\}$.  Then
by Corollary \ref{cor1} $\PP_2^{*}\PP_2$ has eigenvectors $\{e_i\}_{i=1}^N$ with
eigenvalues ``1" for $\{e_i\}_{i=N-k+1}^k$, ``2" for $i=k+1,k+2,\ldots,N$ and
``0"  otherwise.  Hence,
\[ \PP_1^{*}\PP_1+\PP_2^{*}\PP_2 = 2I.\]
\vskip12pt
\noindent {\bf Remark}:  It is worthwhile to note an important property of
the projections we constructed.
By Corollary \ref{cor1}, we have that
\[ \PP_1 e_i = \begin{cases} 0 & \mbox{if}\ \ i=k+1,k+2,\ldots,N\\
 e_i & \mbox{if}\ \ \ i=\lfloor \frac{N}{2}\rfloor, \lfloor \frac{N}{2}\rfloor+1,\ldots,k
 \end{cases} \]
and
\[\PP_2 e_i = \begin{cases} 0 & \mbox{if}\ \ i = 1,2,\ldots,N-k\\
e_i & \mbox{if}\ \  i=\lfloor \frac{N}{2}\rfloor, \lfloor \frac{N}{2}\rfloor+1,\ldots,k
\end{cases}\]
This property carries over to our original subspace $W$ since we can let
\[ V = span\  \{e_i:i=\lfloor \frac{N}{2}\rfloor, \lfloor \frac{N}{2}\rfloor+1,\ldots,k\}.\]
Now we need to see that this works when we use Lemma \ref{lem1}.  But this
is really immediate.  Our original subspace is now $U^*W$ and our projections
are $U^*\PP_1U,U^*\PP_2U$, and so
\[ U^*\PP_1UU^*\PP_2U = U^*\PP_1\PP_2U.\]
\vskip12pt

\noindent {\bf Case 2}:  We have $N=kL$.
\vskip12pt
By Lemma \ref{lem1}, we may assume that our subspace $W$ is:
\[ W = span\ \{e_i+e_{k+i}+e_{2k+i} + \cdots e_{(L-1)k+i}:i=1,2,\ldots,k\} .\]
For $j=1,2,\ldots,L$, let $K_j = \{i=(j-1)k+1,(j-1)k+2,\ldots, jk\}$ and let
$\PP_j$ be the projection onto $W$ along $\{e_i\}_{i\in K_j^c}$.
Then by Corollary \ref{cor1}, $\PP_j^{*}\PP_j$ has eigenvectors $\{e_i\}_{i=1}^N$
with eigenvalues "0" for $i\in K_j^c$ and eigenvalues "L" for $i\in K_j$.
Since
\[ \cup_{j=1}^L K_j = \{1,2,\ldots,N\},\]
and the sets $\{K_j\}_{j=1}^L$ are disjoint, it follows that
 \[ \sum_{i=1}^{L}\PP_i^{*}\PP_i = LI.\]
  \vskip12pt

\vskip12pt


\noindent {\bf Case 3}:  We have $N=kL+M$, $1\le M <k$.
\vskip12pt
We define the subspace $W$ by:
\[ W = span \ \left [ \{\sum_{j=0}^{L}e_{jk+i}:i=1,2,\ldots,M\}\cup
\{\sum_{j=0}^{L-1}e_{jk+i}: j=M+1,M+2,\ldots,k\}\right ] \]
For $j=1,2,\ldots,L$, let $K_j = \{i=(j-1)k+1,(j-1)k+2,\ldots, jk\}$,
and let $K_{L}=\{N-kL+1,N-kL+2,\ldots,N\}$.
Then we get the result by projecting onto $W$ along the sets
$\{K_j^c\}_{j=1}^{k}$.

%
%
%
%
%
%
%
%
%
%
%
%
%
%

\subsection*{Concluding Remarks}
Non-orthogonal fusion frames are clearly natural extensions of orthogonal fusion frames previously introduced \cite{CK04}, \cite{CKL08}.  With non-orthogonal fusion frames, not only can we always make the (new) fusion frame operator $\cSw$ sparse, but also sometimes enable $\cSw$ to become diagonal or tight.  In sensor network data fusion applications, non-orthogonal fusion frames is seen as a flexible tool to resolve the non-sparse nature of the orthogonal fusion frames operator since sensor subspaces and their relationships are given a priori by the sensor physics and the deployment of sensors. Sparsity considerations through non-orthogonal fusion frames seems to be the only effective approach in such applications.  The implementation of the non-orthogonal projections through pseudoframes for subspaces are also discussed in detail.

It is also seen that the flexibility of non-orthogonal fusion frames brings in rather unique and a broad range of properties to the notion of fusion frames.  Our on-going subsequent work includes multi-fusion frame constructions with diagonal or tight $\cSw$, complete tight fusion frame constructions based on one (proper) subspace, classification of positive and self-adjoint operators by projections, and non-orthogonal fusion frames analysis for a given set of subspaces (such as in sensor networks) so that $\cSw$ is either diagonal or tight.  This last task is an ultimate goal.


\end{document}